\documentclass[12pt]{amsart}

\newcommand{\tom}{\to_{M}}

\newcommand{\pdown}{\p_{\down}}

\newcommand{\side}{\te{side}}
\newcommand{\G}{\Gamma}
\newcommand{\simw}{\sim_{w}}

\newcommand{\tow}{\to w}

\renewcommand{\b}{\mathbf{b}}

\newcommand{\R}{\mathcal{R}}

\newcommand{\M}{\ensuremath{\mathcal{M}}}

\newcommand{\dsum}{\di\sum}

\newcommand{\eluw}{\ell(u, w)}

\newcommand{\wtr}{\wt{R}}

\newcommand{\ins}{\in S}

\newcommand{\dep}[1]{{\vrule width 0pt height 0pt depth #1}}

\newcommand{\elx}{\ensuremath{\ell(x)}}
\newcommand{\ely}{\ensuremath{\ell(y)}}

\newcommand{\dlw}{\ensuremath{D_{L}(w)}}
\newcommand{\drw}{\ensuremath{D_{R}(w)}}
\newcommand{\wdlw}{\ensuremath{W_{D_{L}(w)}}}
\newcommand{\wdrw}{\ensuremath{W_{D_{R}(w)}}}

\newcommand{\lew}{\le w}

\newcommand{\elw}{\ensuremath{\ell(w)}}
\newcommand{\elu}{\ensuremath{\ell(u)}}

\newcommand{\elv}{\ensuremath{\ell(v)}}

\newcommand{\elws}{\ensuremath{\ell(ws)}}

\newcommand{\elus}{\ensuremath{\ell(us)}}

\newcommand{\elxy}{\ensuremath{\ell(x, y)}}

\newcommand{\down}{\downarrow}

\newcommand{\eluv}{\ensuremath{\ell(u, v)}}

\newcommand{\wh}{\ensuremath{\widehat}}
\newcommand{\el}{\ensuremath{\ell}}

\DeclareFontFamily{U}{mathx}{\hyphenchar\font45}
\DeclareFontShape{U}{mathx}{m}{n}{
      <5> <6> <7> <8> <9> <10>
      <10.95> <12> <14.4> <17.28> <20.74> <24.88>
      mathx10
      }{}
\DeclareSymbolFont{mathx}{U}{mathx}{m}{n}
\DeclareFontSubstitution{U}{mathx}{m}{n}
\DeclareMathAccent{\widecheck}{0}{mathx}{"71}

\newcommand{\tw}{\textwidth}

\renewcommand{\kill}[1]{}
\newcommand{\dummy}[1]{\mbox{}}

\makeatletter
\newcommand{\xequal}[2][]{\ext@arrow 0055{\equalfill@}{#1}{#2}}
\def\equalfill@{\arrowfill@\Relbar\Relbar\Relbar}
\makeatother

\newcommand{\mto}{\mapsto}

\newcommand{\Set}[2]{\ensuremath{\left\{{#1}\,\middle|\,{#2}\right\}}}

\renewcommand{\k}{\ensuremath{\ol{\mathrm{P}}}}

\newcommand{\h}{\hline}

\renewcommand{\k}[1]{\ensuremath{\left({#1}\right)}}

\newcommand{\ds}{\dots}

\newcommand{\bca}{\begin{cases}}
\newcommand{\eca}{\end{cases}}

\newcommand{\A}{\mathcal{A}}

\newcommand{\mug}{\ensuremath{\infty}}

\newcommand{\Gam}{\Gamma}

\newcommand{\ff}[2]{\ensuremath{\di\fr{#1}{#2}}}

\newcommand{\bpic}{\begin{picture}}\newcommand{\epic}{\end{picture}}

\newcommand{\beda}{\begin{edaenumerate}}
\newcommand{\eeda}{\end{edaenumerate}}

\newcommand{\g}{\ensuremath{\mathbf{g}}}

\newcommand{\cd}{\cdots}

\newcommand{\st}{\strut}

\newcommand{\q}{\quad}

\newcommand{\up}{\uparrow}

\newcommand{\too}{\longrightarrow}

\newcommand{\bq}{\begin{quote}}\newcommand{\eq}{\end{quote}}
\newcommand{\gam}{\gamma}

\newcommand{\sig}{\sigma}

\newcommand{\be}{\begin{enumerate}}\newcommand{\ee}{\end{enumerate}}
\newcommand{\bce}{\begin{center}}\newcommand{\ece}{\end{center}}
\newcommand{\bde}{\begin{description}}\newcommand{\ede}{\end{description}}
\newcommand{\bri}{\begin{flushright}}\newcommand{\eri}{\end{flushright}}
\newcommand{\bb}{\begin{block}}\newcommand{\eb}{\end{block}}
\newcommand{\bt}{\begin{thm}}\newcommand{\et}{\end{thm}}
\newcommand{\bpf}{\begin{proof}}\newcommand{\epf}{\end{proof}}
\newcommand{\bex}{\begin{ex}}\newcommand{\eex}{\end{ex}}
\newcommand{\bexr}{\begin{exr}}\newcommand{\eexr}{\end{exr}}
\newcommand{\bft}{\begin{fact}}\newcommand{\eft}{\end{fact}}
\newcommand{\brk}{\begin{rmk}}\newcommand{\erk}{\end{rmk}}
\newcommand{\ba}{\begin{align*}}\newcommand{\ea}{\end{align*}}
\newcommand{\bexe}{\begin{exe}}\newcommand{\eexe}{\end{exe}}
\newcommand{\tn}{\textnormal}

\newcommand{\bit}{\begin{itemize}}\newcommand{\eit}{\end{itemize}}
\newcommand{\os}{\overset}

\newcommand{\bcm}{}
\newcommand{\ol}{\overline}\newcommand{\ul}{\underline}
\newcommand{\hf}{\hfill}
\newcommand{\ci}{\CIRCLE}
\newcommand{\fr}{\frac}

\newcommand{\nn}{\ensuremath{\mathbf{N}}}
\newcommand{\qq}{\ensuremath{\mathbf{Q}}}

\newcommand{\zz}{\ensuremath{\mathbf{Z}}}

\newcommand{\bd}{\begin{defn}}\newcommand{\ed}{\end{defn}}
\newcommand{\bp}{\begin{prop}}\newcommand{\ep}{\end{prop}}
\newcommand{\p}{\ensuremath{\pi}}
\newcommand{\eh}{\emph}\newcommand{\al}{\alpha}
\newcommand{\sub}{\subseteq}

\newcommand{\fb}{\fbox}
\newcommand{\mb}{\mbox}
\newcommand{\te}{\text}\newcommand{\ph}{\phantom}
\newcommand{\wt}{\widetilde}

\newcommand{\then}{\Longrightarrow}

\newcommand{\di}{\displaystyle}\renewcommand{\a}{\ensuremath{\bm{a}}}
\renewcommand{\b}{\ensuremath{\bm{b}}}\renewcommand{\c}{\ensuremath{\bm{c}}}
\renewcommand{\d}{\ensuremath{\bm{d}}}

\newcommand{\f}{\frac}
\newcommand{\x}{\ensuremath{\bm{x}}}

\newcommand{\np}{\newpage}

\renewcommand{\b}{\beta}
\newcommand{\av}{\tn{av}}

\renewcommand{\a}{\alpha}

\renewcommand{\up}{\uparrow}

\renewcommand{\int}{\in T}

\renewcommand{\x}{\mathbf{x}}

\renewcommand{\d}{\delta}
\renewcommand{\P}{\mathcal{P}}
\newcommand{\Q}{\mathcal{Q}}

\renewcommand{\d}{\delta}

\usepackage[dvipdfmx]{graphicx}
\usepackage[dvipsnames]{xcolor}
\usepackage{float,afterpage}
\usepackage{asymptote,layout}
\usepackage{wrapfig,epic}

\usepackage{geometry}
\usepackage{exscale,latexsym,bm}
\usepackage{amssymb,enumerate,amsmath,amsthm,amsfonts}
\usepackage{verbatim,fancybox,wasysym,fancyhdr,type1cm}
\usepackage[frame,all,poly,curve,knot,arrow]{xy}
\usepackage{colortbl}
\usepackage{boxedminipage}
\usepackage{multirow}

\graphicspath{{./figs/}}
\theoremstyle{definition}
\newtheorem{thm}{Theorem}[section]
\newtheorem{lem}[thm]{Lemma}
\newtheorem{prop}[thm]{Proposition}\newtheorem{cor}[thm]{Corollary}
\newtheorem{cj}[thm]{Conjecture}

\newtheorem{exr}[thm]{Exercise}
\newtheorem{ob}[thm]{Observation}

\newtheorem{ex}[thm]{Example}

\newtheorem{ques}[thm]{Question}

\newtheorem{defn}[thm]{Definition}\newtheorem{rmk}[thm]{Remark}
\newtheorem{fact}[thm]{Fact}
\newtheorem{block}[thm]{}
\newtheorem*{exe}{Exercise}

\newcommand{\orapw}{\ora{\P_{w}}}
\newcommand{\orar}{\ora{R}}
\renewcommand{\g}{\gamma}

\renewcommand{\R}{\mathcal{R}}
\newcommand{\rh}{\rho}
\newcommand{\ora}{\overrightarrow}
\newcommand{\ova}{\overrightarrow}
\newcommand{\olr}{\overrightarrow{R}}

\renewcommand{\tom}{\to_{h}}
\newcommand{\indw}{\tn{ind}_{w}}
\renewcommand{\dep}{\tn{dep}}

\newcommand{\md}{\tn{mid}}

\newcommand{\self}{\tn{self}}

\newcommand{\out}{\tn{out}}

\renewcommand{\a}{\texttt{a}}
\renewcommand{\b}{\texttt{b}}

\renewcommand{\c}{\texttt{c}}
\renewcommand{\d}{\texttt{d}}

\title[]{Weighted counting of Bruhat paths by shifted $R$-polynomials}
\author[Masato Kobayashi]{Masato Kobayashi$^{*}$}
\date{\today}
\address{Department of Engineering\\
Kanagawa University, 3-27-1 Rokkaku-bashi, Yokohama 221-8686, Japan.}
\thanks{*Department of Engineering, Kanagawa University, Japan}
\email{masato210@gmail.com}
\begin{document}
\begin{abstract}
We revisit $R$-polynomials with introducing the new idea ``shifted $R$-polynomials" (or Bruhat weight) for all Bruhat intervals in finite Coxeter groups. 
Then, we apply these polynomials to weighted counting of Bruhat paths. Further, we prove a new criterion of irregularity of lower intervals as analogy of Carrell-Peterson's and Dyer's results.
Also, we present the upper bound of shifted $R$-polynomials for Bruhat intervals of fixed length by Jacobsthal numbers.
\end{abstract}
\maketitle
\tableofcontents
\section{Introduction}\label{s1}

\renewcommand{\A}{\mathbb{A}}
\renewcommand{\M}{\mathbb{M}}

\subsection{Kazhdan-Lusztig polynomials and $R$-polynomials}

The motivation of this article is to better understand 
\eh{Kazhdan-Lusztig} (KL) \eh{polynomials} which they introduced in 1979 \cite{kl}. This is a family of polynomials over nonnegative integer coefficients. Although these polynomials originated from representation theory of Coxeter groups, Hecke algebras and geometry of Schubert varieties, they have been an important topic in algebraic combinatorics as well since then. 
In particular, Bruhat intervals forms a nice subclass of Eulerian posets so that the framework of Eulerian posets ($f$-vector, \a\b-index, $\ds$) works well. 
Here, let us mention the five family of polynomials which play 
some role to investigate KL polynomials:
\begin{itemize}
\item $R$-polynomial
\item $\wt{R}$-polynomial
\item \a\b-, \c\d-index
\item complete \a\b-, \c\d-index
\item Poincar\'{e} polynomial
\end{itemize}
Among these polynomials, only $R$-polynomials have negative coefficients. However, $R$-polynomials satisfy some relations together with KL polynomials:
\[
\dsum_{v\in [u, w]}R_{uv}(q)P_{vw}(q)=q^{\eluw}P_{vw}(q^{-1}).
\]
Thus, it is crucial to better understand coefficients of $R$-polynomials as well. For this reason, we decided to revisit classical $R$-polynomials hoping to find some interpretation by nonnegative integers.
Our idea is simple:we introduce shifted $R$-polynomials (or ``Bruhat weight"); this is just shifting of its variable $q\mto q+1$. We will then show the connection between this shifted $R$-polynomials and $\wt{R}$-polynomials which have nonnegative coefficients so that we can discuss 
weighted counting of Bruhat paths.

%
%


\subsection{Main results}
Main results of this article are the following:
\begin{itemize}
	\item Theorem \ref{th1}: property of Bruhat weight for lower intervals 
		\item Theorem \ref{th2}: another criterion of irregularity of lower intervals
	\item Theorem \ref{th3}: higher Deodhar inequality
	\item Theorem \ref{th4}: the upper bound of shifted $R$-polynomials
	\item Corollary \ref{c1}: 
	the upper bound of Bruhat size of shifted $R$-polynomials
	by Jacobsthal (dihedral) numbers
\end{itemize}
Theorems \ref{th1}, \ref{th4}, Corollary \ref{c1} are new 
while we present Theorems \ref{th2}, \ref{th3} as new interpretations of several known results (Carrell-Peterson, Dyer, the author).

\subsection{organization of this article}

Section \ref{s2} begins the topic with irregularity of Bruhat graph and Poincar\'{e} polynomials.
Section \ref{s3} is all devoted to 
the main discussions on $R$-polynomials, $\wt{R}$-polynomials, shifted $R$-polynomials, and Bruhat weight for edges, Bruhat paths and intervals. Along the way, we provide many examples. 
Section \ref{s4} proves the upper bound of shifted $R$-polynomials as an analogy of the upper bound of $\wt{R}$-polynomials by Fibonacci polynomials.
We end in Section \ref{s5} with recording several ideas for further development of our ideas.



\section{Irregularity of Bruhat graphs}\label{s2}

\subsection{preliminaries on Coxeter groups}

Throughout this article, we denote by $W=(W, S, T, \el,\le)$ a Coxeter system with $W$ the underlying Coxeter group, $S$ its Coxeter generators, $T$ the set of its reflections, $\el$ the length function, $\le$ Bruhat order. Moreover, assume that $W$ is \eh{finite}. 
Unless otherwise noticed, $u, v, w, x, y$ are elements of $W$, $r, s\in S$, $t\in T$ and $e$ is the unit of $W$.
The symbol $\el(u, v)$ means $\el(v)-\elu$ for $u\le v$. 
A \eh{Bruhat interval} is a subposet of $W$ of the form 
\[
[u, w]=\{v\in W\mid u\le v\le w\}.
\]
By $f\le g$ for polynomials $f, g\in\nn[q]$, we mean 
$[q^{i}](f)\le [q^{i}](g)$ for each $i$ where $[q^{i}](P(q))$ denotes the coefficient of $q^{i}$ in a polynomial $P(q)$.

\subsection{Boolean, dihedral posets and Poincar\'{e} polynomial}

The set of all Bruhat intervals forms a subclass of Eulerian posets.
In particular, each lower interval $[e, w]$ is 
Eulerian graded by the length function $v\mto\elv$.

\begin{defn}
The \eh{Poincar\'{e} polynomial} for $w$ is 
\[
\P_{w}(q)=
\sum_{v\le w}q^{\elv}.
\]
\end{defn}
This is the rank generating function of $[e,w]$. 
Observe that 
\[
\P_{w}(-1)=
\sum_{v\le w}(-1)^{\elv}
=
\begin{cases}
	1&w=e,\\
	0&w\ne e.\\
\end{cases}
\]

There are two important classes of Eulerian posets:
 \eh{Boolean} and \eh{dihedral}.
Let $B_{n}$ and $D_{n}$ denote the Boolean and dihedral poset of 
rank $n$, respectively;
we understand that the Boolean or dihedral poset of rank 0 is 
the trivial poset.
Note that $B_{n}=D_n$ for $n=0, 1, 2$ while $B_{n}\ne D_n$ for $n\ge3$.
These posets can be realized as Bruhat intervals (in fact, as lower intervals).
Indeed, Boolean and dihedral intervals are ``extremal" lower intervals 
in the following sense:
\begin{prop}\label{p1}
For any $w$ such that $\elw=n\ge1$, we have 
\[
1+2(q+\cd+q^{n-1})+q^{n}\le \P_{w}(q)\le (1+q)^{n}	
\]
coefficientwise.
In particular, $|D_{n}|=2n\le |[e, w]|\le 2^{n}=|B_{n}|$.
\end{prop}

\begin{proof}
Let $v\in [e, w]$ such that $0<\elv=k<n$. Then 
there exist some $v_{0}, v_{1}\in [e, w]$ such that 
$v_{0}<v<v_{1}$ and $\el(v_{0}, v)=\el(v, v_{1})=1$
since $[e, w]$ is graded.
Now $[v_{0}, v_{1}]$ is an interval of length 2 and 
every such an interval in any Eulerian poset consists of exactly four elements.
So there exists a unique $v'$ such that 
$v_{0}<v'<v_{1}$ and $v'\ne v$. Thus we have 
\[
|\{u\in[e,w]\mid \elu=k\}|\ge 2 
\]
which proves the first inequality. 
To show the second one, choose a reduced word $s_{1}\cd s_{n}$ for $w$.
For each $v\in [e, w]$ with $\elv=k$, there is a reduced subword of this word for $v$ with $n-k$ simple reflections deleted:
\[
v=s_{1}\cd \wh{s_{i_{1}}}\cd \wh{s_{i_{n-k}}}\cd s_{n}
\q\text{(reduced)}.
\]
The number of such words is at most $\binom{n}{n-k}=\binom{n}{k}$.
\end{proof}


\subsection{Bruhat graphs}

\bd{The \emph{Bruhat graph} of $W$ is a directed graph for vertices $w\in W$ and for edges $u\to v$. For each subset $V\sub W$, we can also consider the induced subgraph with the vertex set $V$ (Bruhat subgraph). An edge $u\to v$ is \eh{short} if $\eluv=1$. 
By $a(u, w)$ we mean the directed-graph-theoretic distance from $u$ to $w$.
}\ed 

We can make use of Poincar\'{e} polynomials even for edge counting on Bruhat graphs.
Let $V(w)=[e, w]$ and 
$E(w)=\{u\to v\mid u, v\in [e,w]\}$
 be the vertex and edge set of $[e, w]$, respectively. 
Observe that $|V(w)|=\P_{w}(1)$. 
What is more, each vertex $v\in[e, w]$ is incident to exactly $\elv$ incoming edges so that $|E(w)|=\P_{w}'(1)$ (where $\P'_{w}(q)$ is the (formal) derivative of $\P_{w}(q)$).
It follows from Proposition \ref{p1} that $2\elw\le |V(w)|\le 2^{\elw}$ and 
$\elw^{2}\le|E(w)|\le \elw 2^{\elw-1}$.
In this way, $\P_{w}(q)$ contains subtle information 
on edges of Bruhat graphs on $[e, w]$.
\begin{defn}
The \eh{average} of $\P_{w}(q)$ is $\P_{w}'(1)/\P_{w}(1)$.
\end{defn}
Often, we write $\av\P_{w}(q)=\P_{w}'(1)/\P_{w}(1)$. As seen above, 
\[
\av\P_{w}(q)=\ff{\P_{w}'(1)}{\P_{w}(1)}=\ff{|E(w)|}{|V(w)|}.\]

 
%

%
%
%
%
%

\begin{fact}[Carrell-Peterson \cite{carrell}]
\label{cp}
The following are equivalent:
	\begin{enumerate}
		\item $\av\P_{w}(q)=\elw/2$.\st
		\item $[e, w]$ is regular.
	\end{enumerate}
\end{fact}
\begin{rmk}
Carrell-Peterson (1994) assumed that the Kazhdan-Lusztig polynomial $P_{uw}(q)$ has nonnegative coefficients for all $u\le w$.
This is now (2019 at the time of writing) true due to Elias-Williamson \cite{elias-williamson} in 2014.
\end{rmk}

In fact, $B_{n}$ and $D_{n}$ are both regular. Equivalently, they have same average which is $n/2$.


\subsection{example: 3412}

\begin{ex}
Let $e=1234$ and $w=3412$ in the type $A_{3}$ Coxeter group. 
The lower interval 
$[e, w]$ consists of 14 vertices and 29 edges (Figure \ref{f1}):
\[
\P_{w}(q)=1+3q+5q^{2}+4q^{3}+q^{4},
\]
\[
|V(w)|=\P_w{(1)}=14, \q |E(w)|=\P'_w{(1)}=29,
\]
\[
\av\P_{w}(q)=\ff{29}{14}>2=\ff{\elw}{2}.
\]
Due to Carrell-Peterson, $[e, w]$ is an irregular graph;
Precisely two of 14 vertices, $1234$ and $1324$, have degree 5 while all others have degree $4=\el(w)$.
\end{ex}

\begin{figure}
\caption{the Bruhat graph on $[1234, 3412]$}
\label{f1}
\begin{center}
\resizebox{.95\tw}{!}{
\xymatrix@C=3mm@R=23mm{
&
&&
&&
*+{3412}\ar@{<-}[dr]\ar@{<-}[dlllll]\ar@{<-}[drrrrr]
&&
&&
&\\
*+{3214}\ar@{<-}@/_10ex/[dddrrrrr]&&
&&*+{3142}\ar@{->}[ru]\ar@{<-}[dr]
&&
*+{2413}\ar@{<-}[dlllll]\ar@{<-}[drrr]
&&
&&
*+{1432}\ar@{<-}@/^10ex/[dddlllll]\\
&*+{{2314}}\ar@{->}[lu]\ar@{<-}[drrrr]
&&*+{3124}\ar@{->}[lllu]\ar@{->}[ru]\ar@{<-}[d]
&&
*+{2143}\ar@{->}[ru]&&
*+{1342}\ar@{->}[lllu]\ar@{->}[rrru]\ar@{<-}[d]
&&
*+{{1423}}\ar@{->}[ru]\ar@{<-}[dllll]&\\
&&&*+{2134}\ar@{->}[llu]\ar@{->}[rru]&&
*+{1324}\ar@{->}[llu]\ar@{->}[rru]&&
*+{1243}\ar@{->}[llu]\ar@{->}[rru]&&&\\
&&&&&*+{1234}\ar@{->}[llu]\ar@{->}[u]\ar@{->}[rru]&&&&&\\
}}
\end{center}
\end{figure}
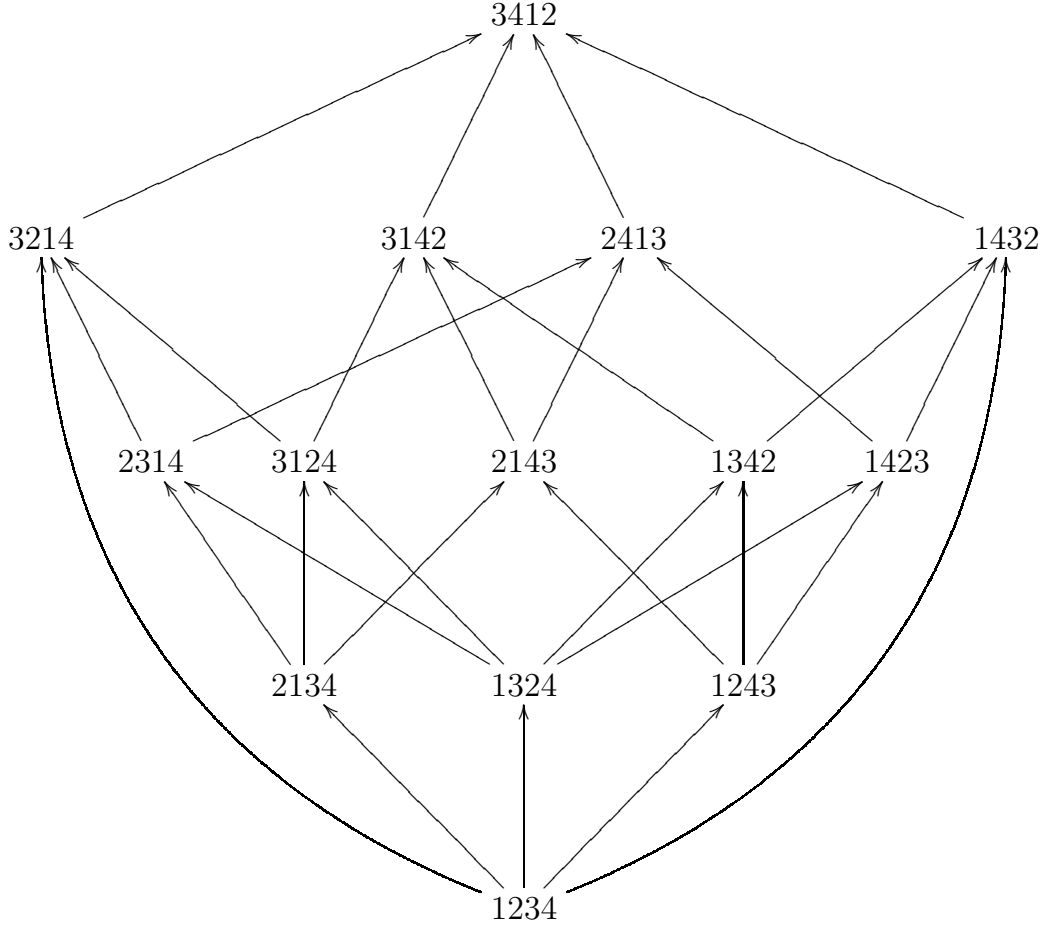

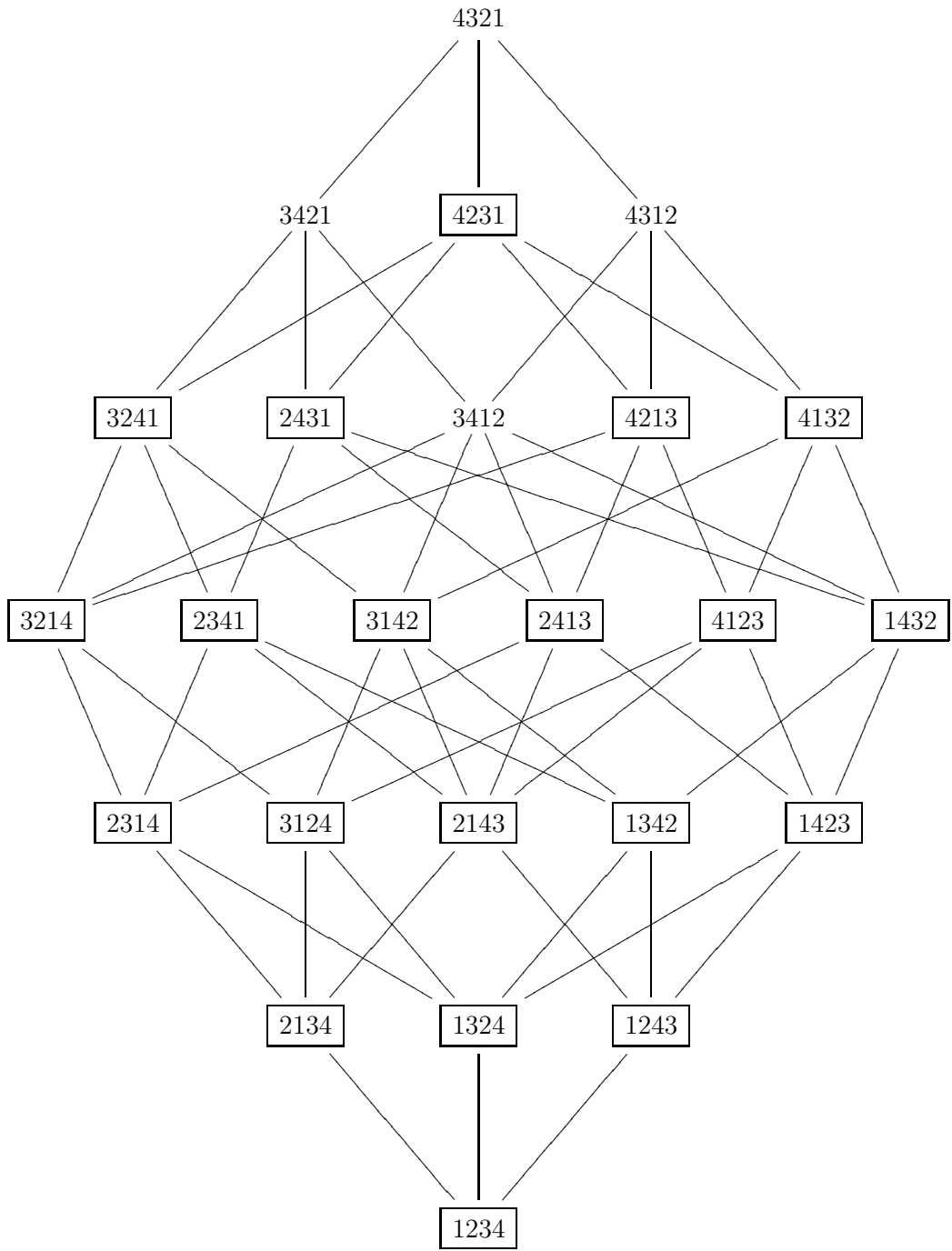
\begin{figure}
\caption{[1234, 4231] in the Hasse diagram of $A_{3}$}
\label{f2}
\begin{center}
\resizebox{.95\tw}{!}{
\xymatrix@C=-1mm@R=23mm{
&&&&&
*+{\fboxrule0pt
\fcolorbox[gray]{0}{1}
{4321}}&&&&&\\
&&&*+{3421}\ar@{-}[rru]\ar@{-}[d]		
&&
*+{\fboxsep5pt\fb{4231}}\ar@{-}[u]\ar@{-}[u]\ar@{-}[dllll]\ar@{-}[drrrr]
&&
*+{4312}\ar@{-}[llu]\ar@{-}[d]	
&&&\\
&*+{{\fboxsep5pt\fb{3241}}}\ar@{-}[rru]\ar@{-}[drrr]
&&*+{{\fboxsep5pt\fb{2431}}}\ar@{-}[rru]\ar@{-}[drrr]\ar@{-}[drrrrrrr]&&
*+{3412}\ar@{-}[llu]\ar@{-}[rru]\ar@{-}[dr]\ar@{-}[dlllll]\ar@{-}[drrrrr]
&&*+{\fboxsep5pt\fb{4213}}\ar@{-}[llu]\ar@{-}[dlllllll]&&
*+{\fboxsep5pt\fb{4132}}\ar@{-}[llu]\ar@{-}[dlllll]&\\
*+{\fboxsep5pt\fb{3214}}\ar@{-}[ru]&&
*+{\fboxsep5pt\fb{2341}}\ar@{-}[lu]\ar@{-}[ru]\ar@{-}[drrr]\ar@{-}[drrrrr]
&&*+{\fboxsep5pt\fb{3142}}\ar@{-}[ru]\ar@{-}[dr]
&&*+{\fboxsep5pt\fb{2413}}\ar@{-}[ru]\ar@{-}[dlllll]\ar@{-}[drrr]
&&*+{\fboxsep5pt\fb{4123}}\ar@{-}[lu]\ar@{-}[dlllll]\ar@{-}[dlll]
\ar@{-}[ru]&&*+{\fboxsep5pt\fb{1432}}\ar@{-}[lu]\\
&*+{\fboxsep5pt\fb{2314}}\ar@{-}[lu]\ar@{-}[ru]\ar@{-}[drrrr]
&&*+{\fboxsep5pt\fb{3124}}\ar@{-}[lllu]\ar@{-}[ru]\ar@{-}[d]
&&*+
{\fboxsep5pt\fb{2143}}\ar@{-}[ru]&&
*+{\fboxsep5pt\fb{1342}}\ar@{-}[lllu]\ar@{-}[rrru]\ar@{-}[d]
&&*+{\fboxsep5pt\fb{1423}}\ar@{-}[lu]\ar@{-}[ru]\ar@{-}[dllll]&\\
&&&*+{\fboxsep5pt\fb{2134}}\ar@{-}[llu]\ar@{-}[rru]&&
*+{\fboxsep5pt\fb{1324}}\ar@{-}[llu]\ar@{-}[rru]&&
*+{\fboxsep5pt\fb{1243}}\ar@{-}[llu]\ar@{-}[rru]&&&\\
&&&&&
*+{\fboxsep5pt\fb{1234}}\ar@{-}[llu]\ar@{-}[u]\ar@{-}[rru]&&&&&\\
}
}
\end{center}
\end{figure}

%

\subsection{monomialization technique}

If we are interested in only the average of a 
Poincar\'{e} polynomial (or more generally a polynomial over nonnegative integer coefficients), there is a useful technique to express it by a monomial as shown below.
Let $\nn$ denote the set of all nonnegative integers and 
$\qq_{\ge0}$ the set of nonnegative rational numbers.

\renewcommand{\A}{\mathbf{W}}
\renewcommand{\M}{\mathbf{M}}
\begin{defn}
\begin{align*}
	\A&=\nn[q^{\qq_{\ge0}}]=\Set{\sum_{i=0}^{d} a_{i}q^{\al_{i}}}{a_{i}, d\in \nn, \al_{i}\in\qq_{\ge0}},
	\\\mathbf{M}&=\Set{aq^{\al}}{a\in \nn, \al\in \qq_{\ge0}}
\end{align*}
We call each element of $\A$ ($\M$) a \eh{weight} (\eh{monomial weight}).
\end{defn}
For example, $q^{\elv}$ is a monomial weight (we call it the \eh{Poincar\'{e} weight} of $v$ for convenience). 

Let $f\in \A$ with $f(1)\ne 0$ (i.e. $f\ne 0$).
Define the \eh{size}, \eh{total}, \eh{average} of $f$ by 
\[
|f|=f(1), \q
\|f\|=f'(1), \q
\av(f)=\ff{\|f\|}{|f|},
\]
respectively.
Set $|0|=\|0\|=0$ and let us not define $\av(0)$.
\begin{prop}
For all $f, g\in \mathbf{W}$, we have the following:
\begin{enumerate}
	\item $|f+g|=|f|+|g|$.
	\item $\|f+g\|=\|f\|+\|g\|$.
	\item $\av(fg)=\av(f)+\av(g)$ ($f, g\ne0$).
\end{enumerate}
\end{prop}

\begin{proof}
We only confirm (3).
\[
\av(fg)=\ff{(fg)'(1)}{(fg)(1)}
=\ff{f'(1)g(1)+f(1)g'(1)}{f(1)g(1)}
=\ff{f'(1)}{f(1)}+\ff{g'(1)}{g(1)}=\av(f)+\av(g).
\]
\end{proof}
\renewcommand{\A}{\mathbf{W}}

\begin{defn}
Define the \eh{monomialization}
$M:\A\to \mathbf{M}$
as follows: set $M(0)=0$. For $f\ne0$, define 
\[
M(f)=|f|q^{\av(f)}.
\]
\end{defn}

As we can easily see, the monomialization preserves size, total and  average:
\[
|M(f)|=|f|, \q
\|M(f)\|=\|f\|, \q
\av(M(f))=\av(f).
\]

\begin{prop}For each $f, g\in \A$, all of the following are true:
\begin{enumerate}
	\item $M(f+g)=M(M(f)+M(g))$.
\item $M(fg)=M(f)M(g)$.
	\item If $f$ is a monomial, then 
$M(f)=f$. In particular, $M(M(f))=M(f)$. 
\end{enumerate}
\end{prop}

\begin{proof}
\begin{align*}
	M(M(f)+M(g))&=M(|f|q^{\av(f)}+|g|q^{\av(g)})
	\\&=(|f|+|g|)q^{(f'+g')/(f+g)}
	\\&=(|f+g|)q^{\av(f+g)}=M(f+g).
\end{align*}
\[
M(fg)=|f||g|q^{\av(fg)}
=|f||g|q^{\av(f)+\av(g)}
=|f|q^{\av(f)}|g|q^{\av(g)}
=M(f)M(g).
\]
(3) is clear.
\end{proof}

\begin{ex}
Let $w=4231$ in $A_{3}$. 
Figure \ref{f2} shows that 
\[
\P_{4231}(q)
=1+3q+5q^{2}+6q^{3}+4q^{4}+q^{5}
=(1+q)^{2}(1+q+2q^{2}+q^{3}).
\]
Then 
\[
M(\P_{4231}(q))
=
M(1+q)^{2}M(1+q+2q^{2}+q^{3})
=(2q^{1/2})^{2}(5q^{8/5})=20q^{52/20}
\]
so that 
\[
\av \P_{4231}(q)=\ff{52}{20}>\ff{50}{20}=\ff{5}{2}=\ff{\el(4231)}{2}.
\]
Again, due to Carrell-Peterson, [1234, 4231] is irregular.
\end{ex}
\begin{rmk}
These examples above come from the characterization of irregular lower intervals in terms of \eh{pattern avoidance}. 
Say a permutation $w$ of $\{1, 2, \ds, n\}$ \emph{contains 3412 \textup{(}4231\textup{)}} if there exist $i, j, k, l$ such that $i<j<k<l$ and $w(k)<w(l)<w(i)<w(j)$ ($w(l)<w(j)<w(k)<w(i)$); 
say $w$ is \eh{singular} if it contains 3412 or 4231.
Then, the following are equivalent:
\begin{enumerate}
	\item $w$ is singular.
	\item $[e, w]$ is irregular.
\end{enumerate}
See Billey-Lakshmibai \cite{billey} for more details on this topic.
\end{rmk}

\section{Weighted counting of Bruhat paths}\label{s3}

\begin{defn}
A \eh{Bruhat path} is a directed path $\G$ such as 
\[
\G:u=v_{0}\to v_{1}\to  \cd \to v_{k}=w
\]
(Below, a ``path" always means a directed path).
Say $\Gam$ is \eh{short} (maximal) if all its edges are short;
it is \eh{long} otherwise. 
For each Bruhat path $\Gam$ as above, we can consider two kinds of \eh{length}:
$k$ is the \eh{absolute length} of $\G$; $\eluw$ is the \eh{Coxeter length}. We write $a(\G)=k$ and $\el(\G)=\el(u, w)$. 
By $x\os{t}{\to} y$ we mean $x\to y$ and $y=xt$, $t\int $.
\end{defn}

\bd{By a \eh{reflection subgroup} of $W$, we mean an algebraic subgroup of $W$  generated by a subset of $T$. 
}\ed
Every reflection subgroup $W'$ is itself a Coxeter system with the canonical generator
\[
\chi(W')=\{t'\in T\mid T_{L}(t')\cap W'=\{t'\}\}
\]
where $T_{L}(t')=\{t\in T\mid \el(tt')<\el(t')\}$. 
A reflection subgroup $W'$ is \eh{dihedral} if $|\chi(W')|=2$.


\bd{Let $<$ be a total order on $T$.
Say $<$ is a \eh{reflection order} if for all dihedral reflection subgroup $W'$ of $W$ with $\chi(W')=\{r, s\}$ $(r\ne s)$, we have 
\[
r<rsr<\cdots <srs <s \te{ or } s<srs<\cdots <rsr <r.\]
}\ed

\subsection{\a\b-, \c\d-index}
Let $\a, \b$ be noncommutative variables 
and $\Gam:u\to{v_1}{}\to\cdots \to{v_k}=w$ a short path. Define 
\newcommand{\tto}{\too}
\[
\x_i=\begin{cases}
	\a&\te{if $t_i<t_{i+1}$}\\
	\b&\te{if $t_i>t_{i+1}$},\\
\end{cases}
\q \psi(\Gamma)=\x_1\cdots \x_k 
\q\te{   and  } \q 
\Psi_{uw}(\a, \b)=\sum_{\Gam}\psi(\Gam)
\]
where the sum is taken over all short paths $\Gam$ from $u$ to $w$.
The \a\b-polynomial $\Psi_{uw}(\a, \b)$ is called the \a\b\eh{-index} of $[u, w]$.
\bft{$\Psi_{uw}(\a, \b)$
is a polynomial of $\a+\b$ and $\a\b+\b\a$.
That is, there exists a unique noncommutative two-variable polynomial $\Phi_{uw}(\c, \d)$ such that 
\[
\Phi_{uw}(\a+\b, \a\b+\b\a)=\Psi_{uw}(\a, \b).
\]
}\eft
The homogeneous \c\d-polynomial $(\deg \c=1, \deg \d=2)$ $\Phi_{uw}(\c, \d)$ is called the \c\d\eh{-index} of $[u, w]$.


\subsection{complete index}
Let $\Gam:u\to{v_1}{}\to\cdots \to{v_k}=w$ be a (not necessarily short)   path. Similarly, define 
\[
\x_i=\begin{cases}
	\a&\te{if $t_i<t_{i+1}$}\\
	\b&\te{if $t_i>t_{i+1}$},\\
\end{cases}
\q 
\wt{\psi}(\Gamma)=\x_1\cdots \x_k
\q \te{  and  } 
\q
\wt{\Psi}_{uw}(\a, \b)=\sum_{\Gam}\wt{\psi}(\Gam)
\]
where the sum is taken over all paths $\Gam$ from $u$ to $w$.
Again, there exists a unique two-variable polynomial $\wt{\Phi}_{uw}(\c, \d)$ such that
\[
\wt{\Phi}_{uw}(\a+\b, \a\b+\b\a)=\wt{\Psi}_{uw}(\a, \b).
\]
$\wt{\Phi}_{uw}(\c, \d)$ is called 
the \eh{complete} \c\d\eh{-index} of $[u, w]$.

\begin{rmk}
These indices do not depend on the choice of a reflection order.
\end{rmk}

\begin{rmk}
In 1990's, the theory on \a\b-, \c\d-index for polytopes and Eulerian posets has been developed by many researchers such as Bayer, Fine, Klapper and Stanley, for example. 
Later Reading \cite{reading} proved (with Karu's work) that all coefficients of $\c\d$-index for a lower interval $[e, w]$ is nonnegative:\,${\Phi}_{ew}(\c, \d)\ge 0$.
 A complete index for a Bruhat interval is a more recent idea in 2010's: See Billera \cite{billera}, Billera-Brenti \cite{billera-brenti}, Blanco \cite{blanco3} and Karu \cite{karu}.
\end{rmk}


\begin{cj}\hf
\begin{enumerate}
	\item Reading \cite{reading}:${\Phi}_{ew}(\c, \d)\le {\Phi}_{B_{\el(w)}}(\c, \d)$.
	\item Billera-Brenti \cite{billera-brenti} strong conjecture:
$\wt{\Phi}_{ew}(\c, \d)\le {\Phi}_{B_{\el(w)}}(\c, \d)$.
\end{enumerate}
\end{cj}
There is one demerit of such indices:
From an \a\b- or a \c\d-monomial 
$\x=\x_{1}\cd \x_{k}$ alone, we cannot recover the Coxeter length of a path. 
Unlike this, we will later on introduce a weight (\eh{Bruhat weight}) which contains some information on both of absolute and Coxeter length of  paths.

\subsection{$R$-polynomials}

Following Bj\"{o}rner-Brenti \cite{bb}, we introduce $R$-polynomials.
\begin{fact}There exists a unique family of polynomials $\{R_{uw}(q)\mid u, w \in W \} \subseteq \mathbf{Z}[q]$ (\emph{$R$-polynomials}) such that
\begin{enumerate}
\item $R_{uw}(q)=0$ \mbox{if $u \not \le w$},
\item $R_{uw}(q)=1$ \mbox{if $u=w$},
\item if $s\in S$ and $\elws<\elw$, then
\begin{align*} R_{uw}(q)=
\begin{cases}
R_{us, ws}(q) &\mbox{ if } \el(us)<\elu,\\
(q-1)R_{u, ws}(q)+qR_{us, ws}(q) &\mbox{ if } \elu<\el(us).
\end{cases}
\end{align*}
\end{enumerate}
\end{fact}

\begin{ex}
$R$-polynomials involve many negative coefficients.
For example, suppose $u\lew$. We can show that 
\[
R_{uw}(q)=
\begin{cases}
	q-1&\eluw=1,\\
	q^{2}-2q+1&\eluw=2,\\
	q^{3}-2q^{2}+2q-1&\eluw=3, u\to w.
\end{cases}
\]
It is tempting to say that 
coefficients of $R$-polynomials alternate in sign.
However, Boe \cite{boe} found the following counterexample:
\[
\resizebox{0.98\hsize}{!}{
$R_{124356, 564312}(q)=1-5q+11q^2-13q^3+8q^4\ul{\,-\,\st}q^5\ul{\st\,-\,}2q^6-q^7+8q^8-13q^9+11q^{10}-5q^{11}+q^{12}.$}
\]
\end{ex}
We wish to understand $R$-polynomials 
as ones over nonnegative integer coefficients with some combinatorial  interpretation. For this purpose, we have to mention Deodhar's work \cite{deodhar} first: he showed that 
$R_{uw}(q)$ ($u\le w$) is the sum of $q^{m}(q-1)^{n}$ 
with $m, n$ nonnegative integers:
\[
R_{uw}(q)=\sum_{
\substack {\ul{\sig}\in \mathcal{D}\\\pi(\ul{\sig})=u}
} q^{m(\ul{\sig})}(q-1)^{n(\ul{\sig})}
\]
where $\mathcal{D}$ is the set of 
\eh{distinguished subexpressions} $\ul{\sigma}$ of 
some fixed reduced expression $s_{1}\cd s_{\elw}$ for $w$ and $\pi$ is a certain map (which we do not need to discuss here).
Shifting the variable $q$ to $q+1$, it is immediate to obtain 
a polynomial of nonnegative integer coefficients.
On the other hand, there is an interesting property of $R$-polynomials 
as characteristic functions of a vertex and an edge at $q=1$ \cite[Chapter 5, Exercise 35]{bb}:
\begin{align*}
	|R_{uw}(q)|&=R_{uw}(1)=\begin{cases}
	1&\mb{if $(u, w)$ is a vertex (i.e. $u=w$), }\\
	0&\tn{ otherwise.}
\end{cases}
	\\\|R_{uw}(q)\|&=
	R'_{uw}(1)=\begin{cases}
	1&\mb{if $(u, w)$ is a directed edge (i.e. $u\tow$),}\\
	0&\tn{ otherwise.}
\end{cases}
\end{align*}
\renewcommand{\G}{\Gamma}
An easy guess is that 
$R$-polynomials are ``counting something" implicitly in Bruhat graphs since vertices and edges are special cases of 
Bruhat paths of absolute length 0 and 1. 
Thus, it is natural to ask if $R$-polynomials somehow count paths of absolute length $\ge 2$. 
Further, if this is the case, then 
its weighting 
should be something like $q^{m}(q-1)^{n}$.
We will see that this guess is right and make this point more explicit after discussing \eh{$\wt{R}$-polynomials} and \eh{shifted $R$-polynomials}.

\begin{rmk}
Caselli \cite{caselli} also proved certain nonnegativity of $R$-polynomials. We have not found any concrete connection yet, though.
\end{rmk}

\subsection{$\wt{R}$-polynomials}

Next, following \cite{bb}, we introduce another family of polynomials associated to $R$-polynomials. They have nonnegative integer coefficients:

\begin{fact}
There exists a unique family of polynomials $\{\widetilde{R}_{uw}(q)\mid u, w \in W \} \subseteq \mathbf{N}[q]$ (\emph{$\widetilde{R}$-polynomials}) such that
\begin{enumerate}
\item $\wt{R}_{uw}(q)=0$ \mbox{if $u \not \le w$},
\item $\wt{R}_{uw}(q)=1$ \mbox{if $u=w$},
\item if $s\in S$ and $\elws<\elw$, then
\begin{align*} \widetilde{R}_{uw}(q)=
\begin{cases}
\widetilde{R}_{us, ws}(q) &\mbox{ if } \elus<\elu,\\
q\widetilde{R}_{u, ws}(q)+\widetilde{R}_{us, ws}(q) &\mbox{ if } \elu<\elus,
\end{cases}
\end{align*}
\item $\wt{R}_{uw}(q)$ ($u \le w$) is a monic polynomial of degree $\ell(u, w)$,\vspace{.05in}
\item $R_{uw}(q)=q^{\ell(u, w)/2}\wt{R}_{uw}(q^{1/2}-q^{-1/2})$.
\end{enumerate}
\end{fact}
We remark that although $q^{1/2}$ and $q^{-1}$ appear 
in the definition above, 
$\wtr_{uw}(q)$ is indeed a polynomial in $q.$
To give a precise description of this family of polynomials, we need 
the following idea:

\begin{defn}
Let $<$ be a reflection order and 
\[
\Gam:u=v_{0}\os{t_{1}}{\to} v_{1}\os{t_{2}}{\to}  \cd \os{t_{k}}{\to} v_{k}=w
\]
a path. 
Say $\G$ is $<$-\eh{increasing} if 
$t_{1}<t_{2}<\cd <t_{k}$. 
We understand that 
any Bruhat path of absolute length 0 or 1 is $<$-increasing for all $<$.
\end{defn}

\begin{fact}[Dyer \cite{dyer2}]
\[
\wt{R}_{uw}(q)=\sum_{\Gam} q^{a(\Gam)}
\]
where the sum is taken all over $<$-increasing paths $\Gam$
 from $u$ to $w$.
Moreover, this sum does not depend on the choice of a reflection order.
\end{fact}

\subsection{Bruhat size and total}


\renewcommand{\g}{\gamma}

\begin{lem}\label{l1}
Let $u<w$,
\label{inc}$a=a(u, w)$ and $\ell=\ell(u, w)$. Then there exist positive integers $\g_\ell\, (=1)$, $\g_{\ell-2}, \dots, \g_a$ such that
\[\wt{R}_{uw}(q)=\g_\ell q^\ell+\g_{\ell-2}q^{\ell-2}+\dots +\g_aq^a.
\]
Consequently, we have
\[R_{uw}(q)=\sum_{i=0}^{\frac{\ell-a}{2}} \g_{a+2i}\,q^{\frac{\ell-a-2i}{2}} (q-1)^{a+2i}.\]
\end{lem}

\bpf{The first statement is a well-known 
property of $\wt{R}$-polynomials. 
As a result, 
\begin{align*}
R_{uw}(q)&=q^{\frac{\ell}{2}}\wt{R}_{uw}(q^{\frac{1}{2}}-q^{-\frac{1}{2}})\\
&=q^{\frac{\ell}{2}}\sum_{i=0}^{\frac{\ell-a}{2}}\g_{a+2i}(q^{\frac{1}{2}}-q^{-\frac{1}{2}})^{a+2i}\\
&=q^{\frac{\ell}{2}}\sum_{i=0}^{\frac{\ell-a}{2}}\g_{a+2i}(q^{-\frac{1}{2}}(q-1))^{a+2i}\\
&=\sum_{i=0}^{\frac{\ell-a}{2}}\g_{a+2i}\,q^{\frac{\ell-a-2i}{2}}(q-1)^{a+2i}.
\end{align*}
}\epf
Hence shifting the variable by one, 
\[
R_{uw}(q+1)=\sum_{i=0}^{\frac{\ell-a}{2}} \g_{a+2i}\,(q+1)^{\frac{\ell-a-2i}{2}} q^{a+2i}
=\sum_{\Gam}(q+1)^{\frac{\ell-a-2i}{2}} q^{a+2i}
\]
is a polynomial of nonnegative integer coefficients. 
It turns out that 
\[
\G\mto (q+1)^{(\ell(\G)-a(\G))/2}q^{a(\G)}\]
is an appropriate choice for a weight of $\G$.
\begin{defn}
Let 
\[
\Gam:u=v_{0}\to v_{1}\to  \cd \to v_{k}=w, \q k=a(\G)
\]
be a Bruhat path.
Define the \eh{Bruhat weight} of $\G$: 
\[
\rho(\G)=(q+1)^{(\el(\G)-a(\G))/2}q^{a(\G)}.
\]
In particular, 
$\rho(\G)$ equals a monomial $q^{k}$ if $\G$ is a short path of length $k$.
\end{defn}

\subsection{Bruhat weight for edges}

Let us introduce a weight also for edges.
The \eh{height} of an edge $u\to v$ is $(\eluv+1)/2$.
In particular, $u\to v$ is \eh{short} if and only if its height is 1; otherwise it is \eh{long}.
 Write $h(u\to v)=\ff{\el(u, v)+1}{2}$.

\begin{defn}
The \eh{Bruhat weight} of an edge of height $h$ is $(q+1)^{h-1}q.$ 
For convenience, we use this symbol:
\[
\to_{h}=(q+1)^{h-1}q.
\]
\end{defn}
This weighting is ``multiplicative" in the following sense:
If $\Gam$ is $v_{0}\to v_{1}\to v_{2}\to \cd \to v_{k-1}\to v_{k}$, 
then 
\[
\rh(\Gam)=
\to_{h(v_{0}\to v_{1})}
\to_{h(v_{1}\to v_{2})}
\cd
\to_{h(v_{k-1}\to v_{k})}.
\]
As we see, $\to_{h}$ is a polynomial of degree $h$.
For example, 
\[
\to_{1}=q,\q
\to_{2}=(q+1)q,\q
\to_{3}=(q+1)^{2}q,
\]
\[
|\to_{1}|=1,\q
|\to_{2}|=2,\q
|\to_{3}|=4 
\te{  and  }
\]
\[
M(\to_{h})=
M(1+q)^{h-1}M(q)=
(2q^{1/2})^{h-1}(q)=
2^{h-1}q^{(h+1)/2}.
\]
\[
\text{with }
|\to_{h}|=2^{h-1}, 
\|\tom\|=2^{h-2}(h+1), 
\av(\to_{h})=\ff{h+1}{2}.
\]


\begin{rmk}
Paths with the same absolute and Coxeter length 
have an identical weight:
For example, 
$\to_{1}\to_{3}=(q+1)^{2}q^{2}=\to_{2}\to_{2}.$
\end{rmk}

\subsection{shifted $R$-polynomials}

\begin{defn}
\eh{The shifted $R$-polynomial} (\eh{$\ora{R}$-polynomial}) for $(u, w)$ is 
\[
\ora{R}_{uw}(q)={R}_{uw}(q+1).
\]
\end{defn}
In particular, 
for $u\not\le w$, 
$\ora{R}_{uw}(q)=0$ 
and 
for $u\le w$, it is monic of degree $\eluw$.
%

\begin{defn}
The \eh{Bruhat size} of $[u, w]$ is $|\olr_{uw}(q)|(= R_{uw}(2))$.
The \eh{Bruhat total} of $[u, w]$ is 
$\|\olr_{uw}(q)\|(=R'_{uw}(2)).$
For convenience, we sometimes write 
\[
|[u, w]|= |\ora{R}_{uw}(q)|, \q \|[u, w]\|= \|\ora{R}_{uw}(q)\|.
\]
\end{defn}

\begin{ex}[Table \ref{t1}]
\begin{align*}
	\ova{R}_{1234, 3421}(q)&=q^{5}+2(q+1)q^{3},
	\\|[1234, 3421]|&=\ova{R}_{1234, 3421}(1)=1^{5}+2\cdot 2\cdot 1^{3}=5,
	\\\|[1234, 3421]\|&=\ova{R}'_{1234, 3421}(1)
={5q^{4}+8q^{3}+6q^{2}\,}
\bigr|_{q=1}
=19.
\end{align*}
\end{ex}

\renewcommand{\wtr}{\wt{R}}

Let us simply say $\ora{R}_{ev}(q)$ is the $\olr$\eh{-polynomial} of $v$. The \eh{Bruhat size} of $v$ is $|v|=|\ora{R}_{ev}(q)|$ and 
the \eh{Bruhat total} of $v$ is 
$\|v\|=\|\ora{R}_{ev}'(q)\|$.
For example, 
$|1234|=1, |3412|=3, |4231|=9, |4321|=11$ (Table \ref{t1}).


\begin{thm}\label{th1}\hf
\begin{enumerate}
	\item $u\le v\then |u|\le |v|$.
	\item $|v|$ is odd.
\end{enumerate}
\end{thm}
\begin{lem}\label{l2}
Let $f, g, h\in \nn[q]$.
\begin{enumerate}
	\item $|u|=2^{\elu/2}\wt{R}_{eu}(2^{-1/2})$.
	\item $g\le h\then fg\le fh\then (fg)(2^{-1/2})\le (fh)(2^{-1/2})$.
	\item $\wt{R}_{eu}(q)q^{\eluv}\le \wtr_{ev}(q)$ if $u\le v$.
In particular, $\wt{R}_{eu}(2^{-1/2})\le 2^{\eluv/2}\wtr_{ev}(2^{-1/2})$.
\end{enumerate}
\end{lem}
\begin{proof}\hf
\begin{enumerate}
	\item Recall that 
\[
	\ora{R}_{eu}(q)=R_{eu}(q+1)=(q+1)^{\elu/2}\wt{R}_{eu}((q+1)^{1/2}-(q+1)^{-1/2}).
\]	
Now	let $q=1$. 
\[
|u|=\ora{R}_{eu}(1)=2^{\elu/2}\wt{R}_{eu}(2^{-1/2}).
\]

	\item Suppose $g\le h$. 
	Then $h-g=\sum_{i=0}^{d} a_{i}q^{i}$ for some nonnegative integers $(a_{i})$. Obviously, $(fh-fg)(q)=f(q)\k{\dsum_{i=0}^{d} a_{i}q^{i}}$ 
	and $q={2^{-1/2}}>0$ yields a nonnegative real number.
	\item Blanco \cite[Theorem 4]{blanco4} proved that 
$u\le x\le v$ $\then$ $\wtr_{ux}(q)\wtr_{xv}(q)\le \wtr_{uv}(q)$.
Let $u\mto e, x\mto u, v\mto v$ so that  
$\wtr_{eu}(q)\wtr_{uv}(q)\le \wtr_{ev}(q)$.
Together with (2) and $q^{\eluv}\le \wtr_{uv}(q)$, we have 
\[
\wtr_{eu}(q)q^{\eluv}\le \wtr_{eu}(q)\wtr_{uv}(q)\le \wtr_{ev}(q).
\]
Finally, set $q=2^{-1/2}$.
\end{enumerate}
\end{proof}

\begin{proof}[Proof of Theorem \ref{th1}]
\hf
\begin{enumerate}
	\item Suppose $u\le v$.
	With the Lemma above, we have 
	\begin{align*}
	|u|&=\ora{R}_{eu}(1)
		= 2^{\elu/2}\wt{R}_{eu}(2^{-1/2})
		\\&\le 2^{\elu/2}\k{2^{\eluv/2}\wt{R}_{ev}(2^{-1/2})}
		=2^{\elv/2}\wt{R}_{ev}(2^{-1/2})=|v|.
	\end{align*}
	\item Let $a=a(e, v)$, $\el=\el(v)$ and $\g_{j}=[q^{j}](\wtr_{ev}(q))$. 
	Because $\g_{\el}=1$, we see that 
\[
|v|=
\ora{R}_{ev}(1)=
\sum_{i=0}^{\frac{\ell-a}{2}} \g_{a+2i}\,2^{\frac{\ell-a-2i}{2}} 1^{a+2i}
=
1+
\sum_{i=0}^{\frac{\ell-a}{2}-1} \g_{a+2i}\,2^{\frac{\ell-a-2i}{2}}
\]
is odd.
\end{enumerate}
\end{proof}

{\renewcommand{\arraystretch}{1.5}
\begin{table}
\caption{$R$-polynomials in $A_{3}$; 
see Billey-Lakshmibai \cite[p.73]{billey}
}
\label{t1}
\begin{center}
	\begin{tabular}{l|l|cccccc}\h
	\mb{}\hf{$v$}\hf&\mb{}\hf{$R_{ev}(q)$}\hf&$|v|$	\\\h
	1234&	$1$&1	\\\h
	1243, 1324, 2134&$q-1$	&1	\\\h
	1342, 1423, 2143, 2314, 3124&$(q-1)^{2}$	&1	\\\h
	1432, 3214&$(q-1)^{3}+q(q-1)$	&3	\\\h
	2341, 2413, 3142, 4123&$(q-1)^{3}$	&1	\\\h
	2431, 3241, 3412, 4132, 4213 &$(q-1)^{4}+q(q-1)^{2}$	&3	\\\h
	4231&$(q-1)^{5}+2q(q-1)^{3}+q^{2}(q-1)$	&9	\\\h
	3421, 4312&$(q-1)^{5}+2q(q-1)^{3}$	&5	\\\h
	4321&$(q-1)^{6}+3q(q-1)^{4}+q^{2}(q-1)^{2}$	&11	\\\h
\end{tabular}
\end{center}
\end{table}}

\subsection{sum of $R$-polynomials}

\begin{fact}
Bruhat order with a reflection order is \eh{Edge-Labeling shellable} (\eh{EL-shellable}) (Dyer \cite{dyer2}):
let $<$ be an arbitrary reflection order.
\begin{enumerate}
	\item For each $[u, w]$, there is a unique $<$-increasing short path from $u$ to $w$, say 
	\[
\G:u=v_{0} \os{t_{1}}{\to} v_{1}\os{t_{2}}{\to} \cd \os{t_{\eluw}}{\to} v_{\eluw}=w.\]
\item Moreover, $(t_{1}, t_{2}, \ds, t_{\eluw})\in T^{\eluw}$ is lexicographically first among all short paths 
from $u$ to $w$.
\end{enumerate}
Consequently, for each $v$, there exists a unique $<$-increasing short path $\G:e\to \ds \to v$ such that 
$\rho(\G)=q^{\elv}$ (corresponding to the Poincar\'{e} weight for $v$).
\end{fact}

\begin{defn}
Define the \eh{Bruhat-Poincar\'{e}} polynomial for $w$:
\[
\orapw(q)=\dsum_{v\le w} 
\ora{R}_{ev}(q).
\]
\end{defn}
Thanks to EL-shellability, 
it splits into two parts:
\[
\orapw(q)=
\sum_{\Gam \te{:short}}\rho(\Gam)
+
\sum_{\Gam \te{:long}}\rho(\Gam)
=
\P_{w}(q)+
\sum_{\Gam \te{:long}}\rho(\Gam).
\]
In particular, $\P_{w}(q)\le \orapw(q)$ 
and $\orapw(-1)=\P_{w}(-1)+0=\P_{w}(-1)$.




\subsection{examples}

\begin{ex}[Figure \ref{f3}]\hf
\begin{enumerate}
	\item $B_{3}:$ Let $s_{1}, s_{2}, s_{3}$ be distinct simple reflections such that they all commute. Let $w=s_{1}s_{2}s_{3}$ so that $[e, w]\cong B_{3}$ as Bruhat graphs.
\[
	\orapw(q)=\sum_{v\le w}\orar_{ev}(q)=
1+3q+3q^{2}+q^{3}=(1+q)^{3}.
\]
\item $D_{3}=[e, w], w=s_{1}s_{2}s_{1}, (s_{1}s_{2})^{3}=e$:
	\[
\orapw(q)=\sum_{v\le w}\ora{R}_{ev}(q)=
1+2q+2q^{2}+(q^{3}+(q+1)q)=
(1+q)^{3}.
\]
\end{enumerate}
\end{ex}
\begin{defn}
Say $[u, w]$ is \eh{Bruhat-Boolean} if 
\[
\sum_{v\in [u, w]}\ora{R}_{uv}(q)=(1+q)^{\eluw}.
\]
\end{defn}
As seen above, $B_{3}$ and $D_{3}$ are both Bruhat-Boolean.

\begin{fact}[{\cite[p.209]{billey}}]The following are equivalent:
	\begin{enumerate}
		\item $[u, w]$ is regular.
		\item Each upper subinterval $[v, w]$ of $[u, w]$ $(v\in [u, w])$ is Bruhat-Boolean.
	\end{enumerate}
\end{fact}
In particular, 
if $[e, w]$ is regular, then $[e, w]$ itself must be Bruhat-Boolean:
\[
\orapw(q)=\dsum_{v\in[e, w]} \ova{R}_{ev}(q)=(1+q)^{\el(w)}. 
\]

\begin{ob}
The finite Coxeter group $W=[e, w_{0}]$ ($w_{0}$ longest element) is Bruhat-Boolean.
In particular, $\sum_{v\in W}|v|=2^{\el(w_{0})}$.
This is because $W=[e, w_{0}]$ is $|\el(w_{0})|$-regular 
and 
\[
\sum_{v\in W}\ora{R}_{ev}(q)
=\sum_{v\le w_{0}}\ora{R}_{ev}(q)=(1+q)^{\el(w_{0})}.
\]
For example, 
$\sum_{v\in A_{4}}|v|=2^{6}=64$.
\end{ob}



\begin{ex}
Let us see two intervals $[u, w]$ such that $u\to w$ and $\eluw=5$. 
One is regular and the other is irregular.
\begin{enumerate}
\item $D_{5}$ (Figure \ref{f3}):
Let $D_{5}=[e, w]$ with $w=s_{1}s_{2}s_{1}s_{2}s_{1}, s_{i}\in S, (s_{1}s_{2})^{5}=e$.
Also, let $u=e$ and $v_{i}, v_{i}'$ be two elements of 
level $i$ $(1\le i\le 4)$. Thus 
$\ova{R}_{uv_{i}}(q)=\ova{R}_{uv'_{i}}(q)$ for such $i$ 
because of combinatorial invariance of 
$R$-polynomials for dihedral intervals (if $[u, w]\cong [x, y]$ as posets and they are both dihedral, 
then $R_{uw}(q)=R_{xy}(q)$ as a consequence of 
the combinatorial invariance of complete \c\d-index for dihedral intervals by Blanco \cite[Lemma 3.2]{blanco3}). 
We can compute the following by induction (see also Table \ref{t2}):
\begin{align*}
	\olr_{uu}(q)&=1,
	\\\olr_{uv_{1}}(q)&=q,
	\\\olr_{uv_{2}}(q)&=q^{2},
	\\\olr_{uv_{3}}(q)&=q^{3}+(q+1)q,
	\\\olr_{uv_{4}}(q)&=q^{4}+2(q+1)q,
	\\\olr_{uw}(q)&=q^{5}+3(q+1)q^{3}+(q+1)^{2}q.
\end{align*}
Altogether, the Bruhat-Poincar\'{e} polynomial of $D_{5}$ is 
\begin{align*}
	\sum_{v\in[u,w]} \olr_{uv}(q)&=
1+2q+2q^{2}+
2(q^{3}+(q+1)q)+2(q^{4}+2(q+1)q^{2})
\\&\ph{==}+(q^{5}+3(q+1)q^{3}+(q+1)^{2}q)
	\\&=(1+q)^{5}.
\end{align*}
\item $[1234, 4231]$ (Figure \ref{f2} and Table \ref{t1}):
\begin{align*}
	\P_{4231}(q)
&=1+3q+5q^{2}+6q^{3}+4q^{4}+q^{5},
	\\\ora{\P}_{4231}(q)&=1+3q+5q^{2}+6q^{3}+4q^{4}+q^{5}
	\\&\ph{==}+\k{2q(q+1)+4q(q+1)q^{2}+2(q+1)q^{3}+(q+1)^{2}q}
	\\&=(1+q)^{3}(1+3q+q^{2}).
	\end{align*}
Although
$M
\k{\ova{\P}_{4231}(q)}
=
(2q^{1/2})^{3}(5q^{5/5})=
40q^{100/40}$ (the average here is $100/40=5/2=\elw/2$), 
$[1234, 4231]$ is \eh{not} Bruhat-Boolean and hence irregular;
$\ora{\P}_{4231}(q)$ is ``slightly larger" than $(1+q)^{5}$.
\end{enumerate}
\end{ex}






\begin{figure}
\caption{Bruhat graph of $B_{3}$, $D_{3}$ and $D_{5}$}
\label{f3}
\begin{center}
\mb{}\hf
\begin{minipage}[c]{.3\tw}
\begin{xy}
0;<5mm,0mm>:
,0+(-12,0)*{\ci}="b1"*+++!U{}
,0+(-12,2)*{\ci}="am"*+++!U{}
,0+(-12,5)*{\ci}="cm"*+++!U{}
,(-3,2)+(-12,0)*{\ci}="al1"*++!R{}
,(3,2)+(-12,0)*{\ci}="ar1"*++!L{}
,(-3,5)+(-12,0)*{\ci}="cl1"*++!R{}
,(3,5)+(-12,0)*{\ci}="cr1"*++!L{}
,(0,7)+(-12,0)*{\ci}="t1"*+++!D{}
,\ar@{->}"b1";"am"
,\ar@{->}"b1";"ar1"
,\ar@{->}"b1";"al1"
,\ar@{<-}"cr1";"am"
,\ar@{<-}"cm";"ar1"
,\ar@{->}"al1";"cl1"
,\ar@{->}"ar1";"cr1"
,\ar@{->}"cl1";"t1"
,\ar@{->}"cr1";"t1"
,\ar@{->}"cm";"t1"
,\ar@{->}"al1";"cm"
,\ar@{->}"am";"cl1"
\end{xy}
\end{minipage}
\hf
\begin{minipage}[c]{.3\tw}
\begin{xy}
0;<5mm,0mm>:
,0+(-12,0)*{\ci}="b1"*+++!U{}
,(-3,2)+(-12,0)*{\ci}="al1"*++!R{}
,(3,2)+(-12,0)*{\ci}="ar1"*++!L{}
,(-3,5)+(-12,0)*{\ci}="cl1"*++!R{}
,(3,5)+(-12,0)*{\ci}="cr1"*++!L{}
,(0,7)+(-12,0)*{\ci}="t1"*+++!D{}
,\ar@{->}"b1";"al1"
,\ar@{->}"b1";"ar1"
,\ar@{->}"al1";"cl1"
,\ar@{->}"ar1";"cr1"
,\ar@{->}"cl1";"t1"
,\ar@{->}"cr1";"t1"
,\ar@{->}"al1";"cr1"
,\ar@{->}"ar1";"cl1"
,\ar@{->}"b1";"t1"
\end{xy}
\end{minipage}
\hf
\begin{minipage}[c]{.3\tw}
\begin{xy}
0;<6mm,0mm>:
,0*{\ci}="b"*+++!U{}
,(-3,2)*{\ci}="l1"*++!R{}
,(3,2)*{\ci}="r1"*++!L{}
,(-3,4)*{\ci}="l2"*++!R{}
,(3,4)*{\ci}="r2"*++!L{}
,(-3,6)*{\ci}="l3"*++!R{}
,(3,6)*{\ci}="r3"*++!L{}
,(-3,8)*{\ci}="l4"*++!R{}
,(3,8)*{\ci}="r4"*++!L{}
,(0,10)*{\ci}="t"*+++!D{}
,\ar@{->}"b";"l1"
,\ar@{->}"b";"r1"
,\ar@{->}"r1";"r2"
,\ar@{->}"r1";"l2"
,\ar@{->}"l1";"r2"
,\ar@{->}"l1";"l2"
,\ar@{->}"l2";"r3"
,\ar@{->}"l2";"l3"
,\ar@{->}"r2";"r3"
,\ar@{->}"r2";"l3"
,\ar@{->}"l3";"l4"
,\ar@{->}"l3";"r4"
,\ar@{->}"r3";"l4"
,\ar@{->}"r3";"r4"
,\ar@{->}"l4";"t"
,\ar@{->}"r4";"t"
,\ar@{->}@/_{3ex}/"l1";"l4"
,\ar@{->}"l1";"r4"
,\ar@{->}"r1";"l4"
,\ar@{->}@/^{3ex}/"r1";"r4"
,\ar@{->}"l2";"t"
,\ar@{->}"r2";"t"
,\ar@{->}"b";"t"
\end{xy}
\end{minipage}\hf
\mb{}
\end{center}

\end{figure}
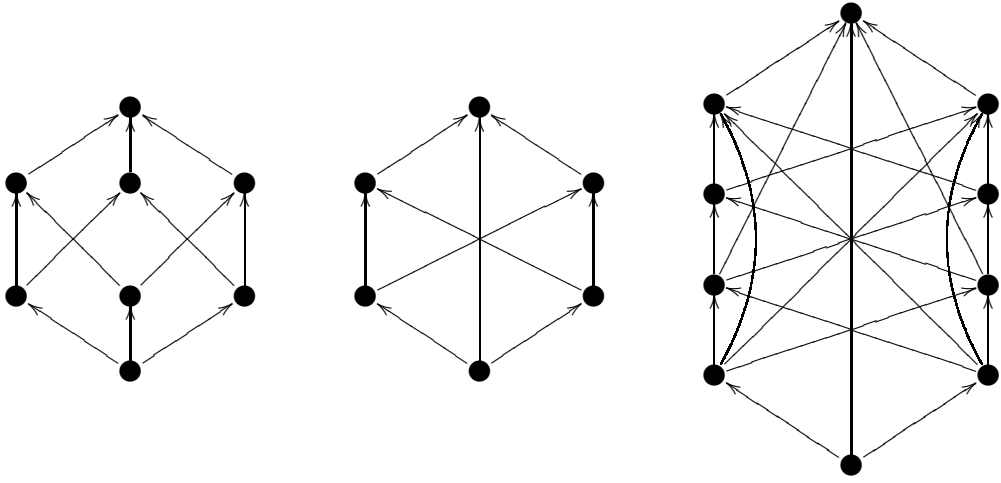





%




%



\subsection{new criterion of irregularity of lower intervals}

Recall that Carrell-Peterson proved 
if $\av(\P_{w}(q))=\ff{|E(w)|}{|V(w)|}\ne \ff{\elw}{2}$, then 
$[e, w]$ is irregular.
We can now generalize this result to 
$\av(\ora{\P}_{w}(q))=\ff{\sum_{v\le w}\|v\|}{\sum_{v\le w}|v|}$.
\begin{thm}\label{th2}
If $\av(\ora{\P}_{w}(q))\ne \ff{\elw}{2}$, then 
$[e, w]$ is irregular.
\end{thm}

\begin{proof}
Suppose $[e, w]$ is regular. 
Then, all upper subintervals of $[e, w]$ are Bruhat-Boolean.
In particular, 
\[
\orapw(q)=\sum_{v\le w} \ora{R}_{ev}(q)=(1+q)^{\elw}.
\]
Therefore, $\av(\orapw(q))=\av((1+q)^{\elw})=\ff{\elw}{2}$. 
\end{proof}

\begin{ex}
\[
\ova{\P}_{3412}(q)=1+3q+5q^{2}+4q^{3}+q^{4}+2(q+1)q
=(1+q)(1+4q+3q^{2}+q^{3})
\]
\[
M(\ova{\P}_{3412}(q))=M(1+q)
M(1+4q+3q^{2}+q^{3})
=
(2q^{1/2})(9q^{13/9})=18q^{35/18}.
\]
Clearly, 
\[
\av\ova{\P}_{3412}(q)=\ff{35}{18}\ne 2=\ff{\elw}{2}.
\]
Therefore, apart from the characterization of singular permutations, 
[1234, 3412] is irregular (note: $\av\ova{\P}_{3412}(q)<2$ while $\av{\P}_{3412}(q)>2$).
\end{ex}

\renewcommand{\f}{\wt{f}}
	
\subsection{higher Deodhar inequality}
For each $[u, w]$, define the following integer sequence $(\f_{i})$:
For each $i$ with $0\le i\le \eluw$, let 
\[
\f_{i}=\f_{i}(u, w)=[q^{i}]\k{\dsum_{v\in [u, w]} \ora{R}_{uv}(q)}
\]
where $[q^{i}](P(q))$ denotes the coefficient of $q^{i}$ in the polynomial $P(q)$.
Clearly, $\f_{0}=1$ since the only $v=u$ term contributes to the constant term and $\ora{R}_{uu}(q)=1$.
What about 
\[
\f_{1}=[q]\k{\sum_{v\in [u, w]} \ora{R}_{uv}(q)}
?
\]
Recall that the weight of a Bruhat path $\G$ is $(q+1)^{(\el(\G)-a(\G))/2}q^{a(\G)}$; only the weight of Bruhat paths involving $q$-term
is one for length 1 (i.e. an edge) with 
\[
[q]\k{(q+1)^{(\el(\G)-1)/2}q}
=1.
\]
Denoting by $\out_{w}(u)$ the out-degree of $u$ in $[u, w]$, that is, 
\[
\out_{w}(u)=|\{v\in [u,w]\mid u\to v\}|, 
\]
we have 
\[
\f_{1}
=
[q]\k{\sum_{v:u\to v\le w} (q+1)^{(\el(u\to v)-1)/2}q}
=
\sum_{v:u\to v\le w}1=
\out_{w}(u).
\]
Thanks to \eh{Deodhar inequality} (Dyer \cite{dyer1}), there is the simple  lower bound of $\f_{1}$ as 
\[
\f_{1}=\out_{w}(u)\ge \eluw.
\]
In fact, this $\ge$ is strict if and only if $[u, w]$ is irregular.
Next, it is natural to ask about $q^{2}$-term:
Paths $\G$ whose weight involving $q^{2}$-term are only ones of absolute length 1 or 2.
Those weights are of the form 
\[
(q+1)^{(\el-1)/2}q
\te{  or }
(q+1)^{(\el-2)/2}q^{2} 
\]
where $\el=\el(\G)$.
Note that 
\[
[q^{2}]\k{(q+1)^{(\el-1)/2}q}
=\ff{\el-1}{2}, \q 
[q^{2}]\k{(q+1)^{(\el-2)/2}q^{2}}
=1.
\]
This leads us to some weighted counting of edges and paths of length 2. Put 
\begin{align*}
	p_{1}&=p_{1}(u, w)=\sum_{u\to v\le w}\ff{\eluv-1}{2}
	=\sum_{u\to v\le w}(h(u, v)-1),
	\\p_{2}&=p_{2}(u, w)=|\{\Gam:u\to v_{1}\to v_{2}\le w,  \text{$<$-increasing}\}|
\end{align*}
so that $\f_{2}=p_{1}+p_{2}$.
\begin{thm}[higher Deodhar inequality]\label{th3}
For all $[u, w]$, we have 
\[
\f_{2}\ge \binom{\eluw}{2}.\] 
Moreover, if $\f_{2}\gneq\binom{\eluw}{2}$, then $[u, w]$ is irregular.
\end{thm}

\begin{proof}
This is a consequence of Kobayashi \cite[Theorem 6.2]{kobayashi} ($x=u$ case).
\end{proof}

\begin{ex}
Let $u=1234, w=3412$.
\begin{align*}
	p_{1}&=\sum_{u\to v\le w} (h(u, v)-1)
	\\&=h(u\to 2134)+h(u\to 1324)+h(u\to 1243)+h(u\to 3214)+h(u\to 1432)
	\\&=0+0+0+1+1=2,
	\\p_{2}&=|\{1342, 1423, 2314, 3124, 3412\}|=4+1=5.
\end{align*}
Therefore, 
\[
\f_{2}=2+5=7>6=\binom{4}{2}.
\]
Again, apart from pattern avoidance, we can now say that $[1234, 3412]$ is irregular.
\end{ex}

%

\section{Lower and upper bounds of shifted ${R}$-polynomials}\label{s4}

In this section, we will prove Theorem \ref{th4} on the sharp lower and upper bounds of $\ova{R}$-polynomials.

\subsection{lower and upper bounds of $\wt{R}$-polynomials}

First, let us review several results on $\wt{R}$-polynomials proved by Brenti \cite[Theorem 5.4, Corollary 5.5, Theorem 5.6]{brenti1}.
\begin{fact}
Let $u\le v$.
\begin{enumerate}
\item $u\le x\le v$ $\then$ $q^{\el(x, v)}\wtr_{ux}(q)\le \wtr_{uv}(q)$.
\item Suppose $W$ is finite. Then $u\le x\le y\le v$ $\then$ $q^{\el(u, x)+\el(y, v)}\wtr_{xy}(q)\le \wtr_{uv}(q)$.
\item Let $x\le y$ in a weak order and $y\le z$. Then, 
$q^{\elxy}\wtr_{yz}(q)\le \wtr_{xz}(q)$.
\end{enumerate}
\end{fact}
These inequalities all follow from the simple fact that 
$\wt{R}_{uw}(q)=q^{\eluw}$ whenever $[u, w]$ is Boolean;
in addition, since each $\wt{R}_{uw}(q)$ is either 0 or monic 
of degree $\eluw$, $q^{n}$ is the least polynomial among 
\[
\{\wtr_{uw}(q)\mid u\le w, \eluw=n\}
\]
in coefficientwise order. 

On the other hand, Fibonacci polynomials 
($F_{0}(q)=1, F_{1}(q)=q, F_{2}(q)=q^{2}, F_{n}(q)=qF_{n-1}(q)+F_{n-2}(q)$ for $n\ge 3$)
give an upper bound of such polynomials. 
In fact, this upper bound is also best possible: $F_{n}(q)$ is the $\wt{R}$-polynomial for any dihedral interval of rank $n$ (Brenti \cite[Proposition 5.3]{brenti2}). Together, there always holds 
\[
q^{n}\le \wt{R}_{uw}(q)\le F_{n}(q)
\]
for $[u, w]$ such that $\eluw=n$. Now it is reasonable to ask 
what corresponds to these inequalities for shifted $R$-polynomials.

\subsection{lower and upper bounds of shifted ${R}$-polynomials}



%

%
%
%
%
Define a sequence of polynomials 
$(d_{n}(q))_{n=0}^{\infty}$ by 
$d_{0}(q)=1, d_{1}(q)=q, d_{2}(q)=q^{2}$
and 
\[
d_{n}(q)=qd_{n-1}(q)+(q+1)d_{n-2}(q)
\q\text{for $n\ge 3$.}
\]
Call $(d_{n}(q))_{n=0}^{\infty}$ \eh{dihedral polynomials}.

It is easy to see that $d_{n}(q)$ is a monic polynomial of degree $n$ and 
morerover it is a weight: $d_{n}(q)\in \A$.
Let $d_{n}=|d_{n}(q)|, 
d'_{n}=\|d_{n}(q)\| $
denote its size and total (Table \ref{t2}).

\begin{lem}\label{l3}
If $[u, w]$ is dihedral, then $\ora{R}_{uw}(q)= d_{\eluw}(q)$.
\end{lem}
\begin{proof}
Induction on $n=\eluw$.
The cases for $n=\eluw\le 2$ 
coincide with Boolean ones: $\ora{R}_{uw}(q)= q^{\eluw}=d_{\eluw}(q)$.
Now suppose	$n=\eluw\ge 3$. Thanks to the combinatorial invariance of $R$-polynomials for dihedral intervals, 
%
we may assume that $[u, w]$ is dihedral, $\el(us)>\elu$ and $\el(ws)<\elw$ for some $s\ins $:
\[
R_{uw}(q)=(q-1)R_{u, ws}(q)+qR_{us, ws}(q),
\]
that is, 
\[
\ora{R}_{uw}(q)=q\ora{R}_{u, ws}(q)+(q+1)\ora{R}_{us, ws}(q).
\]
The inequality $us<ws$ now holds since 
\[
\el(us)=\elu+1\le (\elw-3)+1=\elw-2<\elw-1=\el(ws).
\]
(in a dihedral interval, $x<y \iff \elx<\ely$).
It follows from the property of dihedral intervals that 
subintervals $[us, ws]$ and $[u, ws]$ 
are also dihedral posets of length $n-1, n-2$, respectively.
By inductive hypothesis, $R$-polynomials of those are $d_{n-1}(q)$ and $d_{n-2}(q)$ so that 
\[
\ora{R}_{uw}(q)=
q\ora{R}_{u, ws}(q)+(q+1)\ora{R}_{us, ws}(q)
=qd_{n-1}(q)+(q+1)d_{n-2}(q)=d_{n}(q).
\]
\end{proof}

\begin{thm}\label{th4}
Let $[u, w]$ be a Bruhat interval such that 
$\eluw=n\ge1$. Then, 
\[
q^{n}\le \olr_{uw}(q)\le d_{n}(q)
\te{  and  } 
nq^{n-1}\le \olr'_{uw}(q)\le d'_{n}(q).
\]
Moreover, 
\begin{align*}
	d_{n}(q)&=\ff{q}{q+2}\k{(q+1)^{n}-(-1)^{n}},
	\\d_{n}'(q)&=\ff{2}{(q+2)^{2}}\k{(q+1)^{n}-(-1)^{n}}
	+\ff{nq}{q+2}(q+1)^{n-1}.
\end{align*}
\end{thm}

We give a proof after three lemmas.


\begin{lem}\label{l4}
Let $x<y$. 
Then, there exist some $x', y'$ such that 
$x'\le x, $
$y'\le y, $
$R_{x'y'}(q)=R_{xy}(q)$
 and 
 $\el(x's)>\el(x'), \el(y's)<\el(y')$ for some $s\ins $.
\end{lem}

\begin{proof}
Suppose $x<y$. 
We know that $\ely\ge 1$ implies 
$\el(ys)<\ely$ for some $s\ins $.
If further $\el(xs)>\elx$, then we are done.
Otherwise, $\el(xs)<\elx$. 
Let $x_{1}=xs$, $y_{1}=ys$ ($R_{x_{1}, y_{1}}(q)= R_{xy}(q)$).
Now ask if there exists $s_{1}\ins$ such that $\el(y_{1}s_{1})<\el(y_{1})$ and $\el(x_{1}s_{1})>\el(x_{1})$. If this is the case, then we are done. 
Otherwise, let $x_{2}=x_{1}s_{1}$, $y_{2}=y_{1}s_{1}\dots$. 
This algorithm will end at most $\elx$ steps 
since $\el(x)>\el(x_{1})=\elx-1>\cd>\el(e)=0$ and 
$\el(es)>\el(e)$ for all $s$.
\end{proof}

\begin{lem}\label{l41} 
If $f\le g\le h$ in $\nn[q]$, 
then $f'\le g'\le h'$.
\end{lem}

\begin{proof}
If $f\le g$, then 
$[q^{i}](f')=(i+1)[q^{i+1}](f)\le (i+1)[q^{i+1}](g)=[q^{i}](g)$ for each $i$ which means $f'\le g'$. The same is true for $g$ and $h$.
\end{proof}

To find out a closed formula for $d_{n}(q)$, we take 
the formal power series method.
\begin{lem}\label{l5}
\[
\dsum_{n=0}^{\mug}d_{n}(q)z^{n}=
\ff{1-(q+1)z^{2}}{(1+z)(1-(q+1)z)}.
\]
\end{lem}

\begin{proof}
We wish to find 
\[
D_{q}(z):= \sum_{n=0}^{\mug}d_{n}(q)z^{n}.
\]
First, 
let us compute $D^{\ge 3}_{q}(z):= \dsum_{n\ge 3}d_{n}(q)z^{n}.
$
\begin{align*}
	D^{\ge 3}_{q}(z)&=\dsum_{n\ge 3}(qd_{n-1}(q)+(q+1)d_{n-2}(q))z^{n}
	\\&=qz\dsum_{n\ge 3}d_{n-1}(q)z^{n-1}+
(q+1)z^{2}\dsum_{n\ge 3}d_{n-2}(q)z^{n-2}
	\\&=qz(q^{2}z^{2}+D^{\ge3}_{q}(z))+(q+1)z^{2}(qz+q^{2}z^{2}+D^{\ge3}_{q}(z))
\end{align*}
Thus, 
\[
D^{\ge 3}_{q}(z)=\ff{q(q(q+1)z+q^{2}+q+1)z^{3}}{(1+z)(1-(q+1)z)}
\]
and 
\[
D_{q}(z)=d_{0}(q)+d_{1}(q)z+d_{2}(q)z^{2}
+D^{\ge 3}_{q}(z)
=\ff{1-(q+1)z^{2}}{(1+z)(1-(q+1)z)}.
\]
\end{proof}

\begin{proof}[Proof of Theorem \ref{th4}]
It is easy to check for $n=1, 2$.
Suppose that 
$n=\eluw\ge 3$. By Lemma \ref{l4}, we may assume that $\el(us)>\elu, \el(ws)<\elw$ for some $s$ and 
\[
\ora{R}_{uw}(q)=q\ora{R}_{u, ws}(q)+(q+1)\ora{R}_{us, ws}(q).
\]
By inductive hypothesis, the upper bounds for $\ora{R}$-polynomials of length $\el(u, ws)=n-1, \el(us, ws)=n-2$ intervals are $d_{n-1}(q)$ and $d_{n-2}(q)$ so that 
\[
q^{n}\le \ora{R}_{uw}(q)=q\ora{R}_{u, ws}(q)+(q+1)\ora{R}_{us, ws}(q)
\le qd_{n-1}(q)+(q+1)d_{n-2}(q)=d_{n}(q).
\]
For the second inequalities, just differentiate this as in Lemma \ref{l41}.
Finally, Lemma \ref{l5} implies the last part as follows:
\begin{align*}
	D_{q}(z)&=\ff{1-(q+1)z^{2}}{(1+z)(1-(q+1)z)}
	=\ff{(1-qz-(q+1)z^{2})+qz}{(1+z)(1-(q+1)z)}
	\\&=1+
\ff{qz}{(1+z)(1-(q+1)z)}
	\\&=1-
	\k{\ff{q}{q+2}}
\ff{1}{1+z}
+
\k{\ff{q}{q+2}}
\ff{1}{(1-(q+1)z)}
	\\&=1
-\ff{q}{q+2}
\dsum_{n=0}^{\mug}(-1)^{n}z^{n}
+
\ff{q}{q+2}
\dsum_{n=0}^{\mug}(q+1)^{n}z^{n}
\end{align*}
and hence we conclude that 
\begin{align*}
	d_{n}(q)&=\ff{q}{q+2}\k{(q+1)^{n}-(-1)^{n}},
	\\d_{n}'(q)&=\ff{2}{(q+2)^{2}}\k{(q+1)^{n}-(-1)^{n}}
	+\ff{nq}{q+2}(q+1)^{n-1}
\end{align*}
for $n\ge 1$.
\end{proof}


{\renewcommand{\arraystretch}{1.5}
\begin{table}
\caption{dihedral polynomials and numbers}
\label{t2}
\begin{center}
	\begin{tabular}{c|c|c|ccccc}\h
$n$	&\mb{}\hf$d_{n}(q)$\hf\mb{}&	$d_{n}$&$d_{n}'$\\\h
$0$	&1	&1	&0\\\h
	1&	$q$&1&1	\\\h
	2&$q^{2}$&1	&2	\\\h
	3&$q^{3}+q^{2}+q$&3	&6	\\\h
	4&$q^{4}+2q^{3}+2q^{2}$&5&14		\\\h
	5&$q^{5}+3q^{4}+4q^{3}+2q^{2}+q$	&11&34	\\\h
	6&$q^{6}+4q^{5}+7q^{4}+6q^{3}+3q^{2}$&21&78	\\\h
	7&$q^{7}+5q^{6}+11q^{5}+13q^{4}+9q^{3}+3q^{2}+q$&43&178	\\\h
	8&$q^{8}+6q^{7}+16q^{5}+24q^{4}+22q^{3}+12q^{2}+4q$&85&398	\\\h
	$\vdots$&$\vdots$&$\vdots$&$\vdots$	\\\h
\end{tabular}
\end{center}
\end{table}}





\begin{cor}\label{c1}
Let $[u, w]$ be a Bruhat interval such that $n=\eluw\ge 1$. 
Then 
\[
1\le |[u, w]|\le d_{n} \te{  and  } 
n\le \|[u, w]\|\le d'_{n}.
\]
Moreover, 
\[
d_{n}=\ff{\,1\,}{3}(2^{n}-(-1)^{n})
\te{  and  } 
d_{n}'=\ff{\,2\,}{9}\k{2^{n}-(-1)^{n}+3n\cdot 2^{n-2}}\]
\end{cor}



The sequence $(J_{n})_{n=0}^{\mug}$ with 
$J_{0}=0, J_{1}=1$ and 
\[J_{n}=J_{n-1}+2J_{n-2} \q n\ge 2\]
is known as \eh{Jacobsthal sequence} (The On-line Encyropedia of Integer Sequences A001045 \cite{oeis}) in combinatorics and number theory. 
The only difference between $J_{n}$ and 
our $d_{n}$ is the initial value: 
$d_{0}=1\ne 0=J_{0}$.
For $n\ge 1$,
\[
d_{n}=J_{n}=\ff{\,1\,}{3}\k{2^{n}-(-1)^{n}}.
\]

\section{Concluding remarks}\label{s5}

We end with recording several ideas for our future research.

\subsection{double $R$-polynomials}

Let $p, q$ be commutative variables.
\begin{defn}
Define the \eh{double $R$-polynomial} for $(u, w)$ by 
\[
R_{uw}(p, q)=\sum_{i=0}^{\frac{\ell-a}{2}} \gam_{a+2i} p^{(\el-a-2i)/2}(q-1)^{a+2i}
\]
where $(\g_{j})$ are positive integers such that 
\[R_{uw}(q)=\sum_{i=0}^{\frac{\ell-a}{2}} 
\g_{a+2i}\,q^{\frac{\ell-a-2i}{2}} (q-1)^{a+2i}.\]
as in Lemma \ref{l1}.
\end{defn}


Many polynomials in this articles are disguises of this 
double $R$-polynomials.
\begin{ob}
\begin{align*}
	R_{uw}(q, q)&=R_{uw}(q),
	\\R_{uw}(q+1, q+1)&=\ova{R}_{uw}(q),
	\\R_{uw}(1, q+1)&=\wt{R}_{uw}(q),
	\\R_{uw}(0, q+1)&=q^{\eluw}.
\end{align*}
\end{ob}

%
%
%
%
%

%
\begin{ques}
What is the recurrence of double $R$-polynomials?
\end{ques}





\subsection{Bruaht size on Bruhat graph}

We proved that 
$u\le v\then |u|\le |v|$. 
Since Bruhat order is the transitive closure of edge relations, 
it is reasonable to ask this:
\begin{ques}
Suppose $u\to v$. When $|u|\lneq |v|$ and when not?
\end{ques} 
It is probably the easiest to try the type A case first.
\subsection{extension of higher Deodhar inequality}
We showed that for each interval $[u, w]$, we have 
\[
\f_{i}(u, w)\ge \binom{\eluw}{i}
\]
for $i=0, 1, 2$. We do not know if the similar inequalities hold 
for all $i\ge 3$. If this is the case, then we always have 
\[
\dsum_{v\in [u, w]}\ora{R}_{uv}(q)\ge (1+q)^{\eluw}
\]
which looks very nice. Prove or disprove it.

\end{document}